\documentclass[10pt,draft]{amsart}
\usepackage{amsmath,amssymb,amsthm}
\begin{document}
\newtheorem{lem}{Lemma}[section]
\newtheorem{prop}{Proposition}[section]
\newtheorem{cor}{Corollary}[section]
\numberwithin{equation}{section}
\newtheorem{thm}{Theorem}[section]

\theoremstyle{remark}
\newtheorem{example}{Example}[section]
\newtheorem*{ack}{Acknowledgments}

\theoremstyle{definition}
\newtheorem{definition}{Definition}[section]

\theoremstyle{remark}
\newtheorem*{notation}{Notation}
\theoremstyle{remark}
\newtheorem{remark}{Remark}[section]

\newenvironment{Abstract}
{\begin{center}\textbf{\footnotesize{Abstract}}%
\end{center} \begin{quote}\begin{footnotesize}}
{\end{footnotesize}\end{quote}\bigskip}
\newenvironment{nome}

{\begin{center}\textbf{{}}%
\end{center} \begin{quote}\end{quote}\bigskip}

\newcommand{\triple}[1]{{|\!|\!|#1|\!|\!|}}

\newcommand{\xx}{\langle x\rangle}
\newcommand{\ep}{\varepsilon}
\newcommand{\al}{\alpha}
\newcommand{\be}{\beta}
\newcommand{\de}{\partial}
\newcommand{\la}{\lambda}
\newcommand{\La}{\Lambda}
\newcommand{\ga}{\gamma}
\newcommand{\del}{\delta}
\newcommand{\Del}{\Delta}
\newcommand{\sig}{\sigma}
\newcommand{\ome}{\omega}
\newcommand{\Ome}{\Omega}
\newcommand{\C}{{\mathbb C}}
\newcommand{\N}{{\mathbb N}}
\newcommand{\Z}{{\mathbb Z}}
\newcommand{\R}{{\mathbb R}}
\newcommand{\T}{{\mathbb T}}
\newcommand{\Rn}{{\mathbb R}^{n}}
\newcommand{\Rnu}{{\mathbb R}^{n+1}_{+}}
\newcommand{\Cn}{{\mathbb C}^{n}}
\newcommand{\spt}{\,\mathrm{supp}\,}
\newcommand{\Lin}{\mathcal{L}}
\newcommand{\SSS}{\mathcal{S}}
\newcommand{\F}{\mathcal{F}}
\newcommand{\xxi}{\langle\xi\rangle}
\newcommand{\eei}{\langle\eta\rangle}
\newcommand{\xei}{\langle\xi-\eta\rangle}
\newcommand{\yy}{\langle y\rangle}
\newcommand{\dint}{\int\!\!\int}
\newcommand{\hatp}{\widehat\psi}
\renewcommand{\Re}{\;\mathrm{Re}\;}
\renewcommand{\Im}{\;\mathrm{Im}\;}
\def\11{{\rm 1~\hspace{-1.4ex}l} }

\title
[An Improvement on the  Br\'ezis-Gallouet technique ]
{An Improvement on the Br\'ezis-Gallou\"et technique
for 2D NLS and 1D half-wave equation}
\author[Tohru Ozawa]{Tohru~Ozawa}
\author[Nicola Visciglia]{Nicola Visciglia}
\address{Department of Applied Physics, Waseda University, Tokyo 169-8555, Japan}\email{txozawa@waseda.jp}
\address{Dipartimento di Matematica Universit\`a di Pisa ,
Largo B. Pontecorvo 5, 56127 Pisa. Italy}\email{ viscigli@dm.unipi.it}
\maketitle
\begin{abstract}
We revise the classical approach by Br\'ezis-Gallou\"et 
to prove global well posedness for nonlinear evolution equations. 
In particular we prove global well--posedness 
for the quartic NLS posed on general domains $\Omega$ in $\R^2$ with initial data in $H^2(\Omega)\cap H^1_0(\Omega)$, and for the quartic nonlinear half-wave equation on $\R$ with initial data in $H^1(\R)$.
\end{abstract}
The main aim of this paper is to revise the technique developed
by Br\'ezis-Gallou\"et to study the global well--posedness 
of Cauchy problems associated with some nonlinear evolution equations. 
We prove that by the Br\'ezis-Gallou\"et technique applied to higher order energy with integration by parts, the standard theory developed in \cite{BG}
and \cite{KLR} for NLS and half-wave equation with cubic nonlinearity, has an improvement to quartic nonlinearity.

Our first result
concerns an extension
to higher order nonlinearities of the very classical result in \cite{BG}.
More precisely the first family of problems that we shall address is the following
one:
\begin{equation}\label{NLS}
\begin{cases}
i\partial_t u + \Delta u=\lambda u|u|^{3}, (t, x)\in \R\times \Omega,
\\
u(t, x)=0, (t, x)\in \R \times \partial \Omega,
\\
u(0)=\varphi,
\end{cases}
\end{equation}
where $\lambda=\pm 1$, $\Omega\subset \R^2$ is open and satisfies the following 
hypothesis:
$$\leqno (H1) \hbox{ } \exists \hbox{
$L\in {\mathcal L}(H^2(\Omega), H^2(\R^2)) \cap {\mathcal L}(H^1(\Omega), H^1(\R^2))$
s.t. $(Lu)_{|\Omega}=u$ a.e. in $\Omega$};$$
$$ \leqno(H2) \hbox{ } L^2(\Omega)\supset H^2(\Omega)\cap H^1_0(\Omega)\ni u\mapsto \Delta u \in L^2(\Omega) \hbox{ is self-adjoint }.$$
By the celebrate Br\'ezis-Gallou\"et inequality it follows that if $\Omega$ satisfies $(H1)$,
then
the following logarithmic Sobolev embedding occurs:
\begin{equation}\label{improved}
\|v\|_{L^\infty(\Omega)} \lesssim \|v\|_{H^1(\Omega)} \sqrt {\ln \big (2+ \|v\|_{H^2(\Omega)}\big )} + 1, \forall v\in H^2(\Omega).
\end{equation}

There has been a growing interest in the last decades on the Cauchy problem
associated with NLS on domains, starting from the pioneering paper \cite{BG}.
In this paper the authors can deduce global well--posedness 
for the defocusing cubic NLS on domains $\Omega \subset \R^2$,
by combining \eqref{improved}
with the conservation of the energy.
A first extension of the result by Br\'ezis-Gallou\"et,
up to the fourth order  nonlinearity, was obtained in \cite{T} under some restrictive
conditions on the initial data $\varphi$. More precisely it is assumed $\varphi|\varphi|\in H^3(\Omega)\cap H^1_0(\Omega), \Delta \varphi \in H^1_0(\Omega)$. 
A fundamental tool to treat NLS on domains, with higher order nonlinearities,
are the so called Strichartz inequalities
(see \cite{Ca} and the bibliography therein for the case $\Omega=\R^2$).
In \cite{BGT} it is proved a suitable version of Strichartz inequalities
with loss, on general compact manifolds. Beside other results in this paper it is studied the
Cauchy problem associated with  NLS on 2D compact manifolds for every nonlinearity  $u|u|^p$. 
The results in \cite{BGT} have been extended to NLS 
on domains $\Omega\subset \R^2$, under suitable assumptions. In particular the case of bounded domains and external domains
has been widely investigated in the literature. 
Just to quote a few of those results
we mention
\cite{A}, \cite{BBS}, \cite{BGT2}, \cite{I}.... 

Due to the huge literature devoted to NLS on 2D domains, Theorem \ref{nls} below could be considered 
somewhat weaker compared with
the known results, however
we prefer to keep its statement along this paper for three reasons. First of all our argument
is exclusively based on integration by parts and energy estimates, and hence it is independent on the use of Strichartz estimates. The second reason is that
the proof of Theorem \ref{nls} can help to understand the idea
behind the more involved
proof of our second result
concerning the nonlinear half--wave equation, where as far as we know
our result is a novelty in the literature.
The third reason is that as far as we know it is unclear whether or not the aforementioned Strichartz estimates are available under the rather general assumptions
$(H1), (H2)$.
\\ 

Let us recall that by the usual energy estimates, in conjunction
with the classical Sobolev embedding $H^2(\Omega)\hookrightarrow L^\infty(\Omega)$,
one can prove that the Cauchy problem \eqref{NLS} is well posed locally in time
provided that $\varphi\in H^2(\Omega) \cap H^1_0(\Omega)$.
More precisely there exists one unique solution $u\in {\mathcal C}([0, T_{max}); H^2(\Omega)\cap H^1_0 (\Omega))$
of \eqref{NLS},  where $T_{max}>0$. Moreover we have the alternative: 
either $T_{max}=\infty$ or $T_{max}<\infty$
and $\lim_{t\rightarrow T_{max}^-} \|u(t)\|_{H^2(\Omega)}=\infty$.\\

The first result of the paper is the following.
\begin{thm}\label{nls} Let  
$\Omega\subset \R^2$ be an open set that satisfies $(H1), (H2)$, 
$\varphi\in H^2(\Omega) \cap H^1_0(\Omega)$ and let
$u\in {\mathcal C}([0, T_{max}); H^2(\Omega) \cap H^1_0(\Omega))$ 
be the unique local solution of \eqref{NLS}. Then we have the following alternative:
either $T_{max}=\infty$ or  $T_{max}<\infty$ and $\sup_{[0, T_{max})} \|u(t)\|_{H^1(\Omega)}=\infty$.
\end{thm}

Next we give some concrete conditions on the initial data $\varphi$ in order to guarantee
global well--posedness of \eqref{nls}.
We need to introduce the energy preserved
along \eqref{nls} for $\lambda=\pm 1$:
\begin{equation}\label{enNLS}E_{NLS, \pm}(u)=\frac 12 \int_{\Omega} |\nabla u|^2 
dx \pm \frac 1 5 \int_{\Omega} |u|^5 dx.\end{equation}
We also introduce the ground state $Q(|x|)$ defined as the unique solution
to
$$-\Delta Q+Q=Q^4, \quad  Q\in H^1(\R^2), \quad Q>0.$$

We are now in a position to state the following global well--posedness result.
\begin{cor}\label{gwp}
Let $\Omega$ be as in Theorem \ref{nls} and $\varphi\in H^2(\Omega) \cap H^1_0(\Omega)$.\\
If $\lambda=1$ then \eqref{NLS} has one unique
global solution $u\in {\mathcal C}([0, \infty); H^2(\Omega) \cap H^1_0(\Omega))$.\\
If $\lambda=-1$ and $\varphi$ satisfies: 
\begin{align}\label{deu}E_{NLS,-} (\varphi)  \|\varphi\|_{L^2}^4 < E_{NLS, -}(Q) 
\|Q\|_{L^2}^4\\\nonumber
\hbox{ and } \|\nabla  \varphi\|_{L^2}  \|\varphi\|_{L^2(\Omega)}^2 < \|\nabla  Q\|_{L^2} 
\|Q\|_{L^2}^2,\end{align}
then \eqref{nls} has one unique
global solution $u\in {\mathcal C}([0, \infty); H^2(\Omega) \cap H^1_0(\Omega))$.
\end{cor}
The proof of  Corollary \ref{gwp} follows by Theorem \ref{nls}
in conjunction with the conservation of the energy \eqref{enNLS}.
In fact in the defocusing case, since the energy is positive definite,
it prevents blow--up of the $H^1$ norm. In the focusing case
a combination of the conservation of the energy with 
conditions \eqref{deu}, prevents blow--up of the $H^1$-norm 
via a standard continuity argument (see 
\cite{HR} for details). 
\\

The second family of Cauchy problems that we consider in this paper is associated
with the fourth order
nonlinear half-wave equation:
\begin{equation}\label{NLSsemirel}
\begin{cases}
i\partial_t u -  |D_x| u=\lambda u|u|^{3}, (t, x)\in \R\times \R,\\
u(0)=\varphi \in H^1(\R),
\end{cases}
\end{equation}
where $|D_x|=\sqrt{-\partial_x^2}$ is the first order non-local fractional derivative,
$\lambda=\pm 1$. Let us mention that evolution problems with nonlocal dispersion arise in various physical settings
(see \cite{ES}, \cite{MMT}, \cite{FL}, \cite{KLS}). In the case of a cubic nonlinearity,
the Cauchy problem
\eqref{NLSsemirel} 
is strictly related with the Szeg\"o model (see \cite{GG}, \cite{P}).\\
\\
We recall that by standard arguments one can prove 
the existence of one unique solution $u\in {\mathcal C}([0, T_{max}); H^1(\R))$
of \eqref{NLSsemirel},  where $T_{max}>0$. Moreover we have the alternative: 
either $T_{max}=\infty$ or $T_{max}<\infty$
and $\lim_{t\rightarrow T_{max}^-} \|u(t)\|_{H^1(\R)}=\infty$.\\

We can state our second result.
\begin{thm}\label{nlssemirel}
Let $u\in {\mathcal C}([0, T_{max}); H^1(\R))$ 
be the unique local solution of \eqref{NLSsemirel}. Then we have the following alternative:
either $T_{max}=\infty$ or $T_{max}<\infty$ 
and $\sup_{[0, T_{max})} \|u(t)\|_{H^\frac 12(\R)}=\infty$.
\end{thm}
Next we give some concrete conditions on the initial data $\varphi$
in order to guarantee
global well--posedness of \eqref{NLSsemirel}.
We need to introduce the energy preserved
along \eqref{NLSsemirel} for $\lambda=\pm 1$:
\begin{equation}\label{enNLSsemir}E_{HW, \pm}(u)=\frac 12 \int_{\R} ||D_x|^\frac 12 u|^2 dx \pm \frac 1 5 \int_{\R} |u|^5 dx.\end{equation}
We also introduce $R\in H^\frac 12(\R)$ as the unique (non trivial) optimizer
of the following Gagliardo-Nirenberg
inequality
\begin{equation}\label{GNdispintro} 
\|f\|_{L^5(\R)} \leq C_{GN} \||D_x|^\frac 12 f\|_{L^2(\R)}^\frac 35\|f\|_{L^2(\R)}^\frac 25,
\end{equation}
that satisfies
\begin{equation}\label{normal}|D_x| R + R=R^4, \quad R(x)=R(|x|)>0.
\end{equation}
The uniqueness of $R$ defined as above is proved in \cite{RFEN}
(concerning a general proof
on the existence of optimizers for Gagliardo-Nirenberg inequalities see \cite{BFV}).\\

The next result is a version Corollary \ref{gwp} in the context of the half-wave equation. 
\begin{cor}\label{gwpHW}
Assume $\lambda=1$ then \eqref{NLSsemirel} has one unique
global solution $u\in {\mathcal C}([0, \infty); H^1(\R))$.\\
Assume $\lambda=-1$ and $\varphi$ satisfies: 
\begin{align}\label{deusemir}E_{HW,-} (\varphi)  \|\varphi\|_{L^2}^4 < E_{HW, -}(R) 
\|R\|_{L^2}^4\\\nonumber
\hbox{ and } \||D_x|^\frac 12\varphi\|_{L^2}  \|\varphi\|_{L^2(\R)}^2 < \||D_x|^\frac 12 
R\|_{L^2}  \|R\|_{L^2}^2,\end{align}
then \eqref{NLSsemirel} has one unique
global solution $u\in {\mathcal C}([0, \infty); H^1(\R))$.
\end{cor}
Along the paper we shall present a proof of  Corollary \ref{gwpHW}. 
Of course in the defocusing case it follows by Theorem \ref{nlssemirel}
in conjunction with the fact that the energy $E_{HW,+}$
is positive definite.
In the focusing case the proof is more involved and we need to adapt the argument in \cite{HR} in a non-local context.

The global well--posedness results above can be considered 
as an extension to the quartic half--wave equation
of part of the results proved by Krieger-Lenzmann-Raphael in \cite{KLR}.
In this paper in fact the authors treat, beside very interesting blow-up results, 
the Cauchy theory for the half-wave equation with cubic nonlinearity via the classical approach
in \cite{BG}. We should also notice that in \cite{KLR} the authors work
in $H^\frac 12(\R)$, while in Theorem \ref{nlssemirel} 
we work in $H^1(\R)$.\\
\\
A basic tool along the proof of Theorem \ref{nlssemirel}
will be the following version of \eqref{improved}:
\begin{equation}\label{improvedsemirel}
\|v\|_{L^\infty(\R)} \lesssim \|v\|_{H^{\frac 12}(\R)} \sqrt {\ln \big (2+ 
\|v\|_{H^1(\R)}\big )} +1 , \forall v\in H^1(\R).
\end{equation}
Its proof follows by a straightforward adaptation of the argument
in \cite{BG}. Hence we skip it and we shall make an extensive use of \eqref{improvedsemirel}
without any further comment.
\\
\\
{\bf Acknowledgment}: {\em N.V. is supported by the FIRB project Dinamiche Dispersive. The authors are grateful to N. Tzvetkov for interesting comments.}
\section{Proof of Theorem \ref{nls}}
Along this section we use the notations:
$$\nabla u= (\partial_x u, \partial_y u), \Delta=\partial_x^2+\partial_y^2,
(f, g)=\int_{\Omega} f\cdot \bar g \hbox{ } dx,
L^2=L^2(\Omega), H^k=H^k(\Omega).$$
We also introduce the following energy:
$${\mathcal E}(u) = \|\Delta u\|_{L^2}^2 -2\lambda {\mathcal Re} (\Delta u, u|u|^3)
-\frac 34 \lambda (\nabla |u|^2, \nabla(|u|^{2}) |u|).$$
\begin{lem}
Let $u$ be as in Theorem \ref{nls}, then we have the following identity:
\begin{align}\frac d{dt} ({\mathcal E}(u) + \|u\|_{L^2}^2)&=
-2 \lambda^2 {\mathcal Im} (\nabla u, u \nabla (|u|^6))
\\\nonumber&+ \frac{3}4 \lambda 
(|\nabla |u|^2|^2, \partial_t |u|))+2\lambda (|\nabla u|^2, \partial_t (|u|^{3})).
\end{align}
\end{lem}
\begin{proof}
Recall that $\frac d{dt} \|u\|_{L^2}^2=0$,
hence we shall treat $\frac d{dt} {\mathcal E}(u)$. Next we assume
that the solution is regular enough in order to justify all the computations.
In the case that the solution $u$ is only $H^2$, then one  can proceed by a smoothing argument via the
Yosida regularization (we skip this technical but standard regularization argument).\\
We start with the following computation:
\begin{align*}
\frac d{dt} &\|\Delta u\|_{L^2}^2= 2 {\mathcal Re} (\Delta \partial_t u, \Delta u)
=2 {\mathcal Re} (\Delta \partial_t u, -i \partial_t u + \lambda u|u|^3)\\\nonumber&=2\lambda  {\mathcal Re} (\Delta \partial_t u, u|u|^3)=2\lambda  \frac d{dt} {\mathcal Re} (\Delta u, u|u|^3)
- 2\lambda  {\mathcal Re} (\Delta u, \partial_t (u|u|^3)),\end{align*}
where we used the equation solved by $u$ in the second equality.
Next notice that
\begin{align*}
{\mathcal Re} &(\Delta u, \partial_t (u|u|^3))=
{\mathcal Re} (\Delta u, \partial_t u |u|^3) + {\mathcal Re} (\Delta u, u\partial_t (|u|^3))
\\\nonumber
&={\mathcal Re} (\Delta u, \partial_t u |u|^3) + \frac 12 (\Delta |u|^2, \partial_t (|u|^3) - (|\nabla u|^2, \partial_t (|u|^3)) =I+II+III.\end{align*}
By using the equation solved by $u$ we get
$$I=  {\mathcal Re} (\Delta u, -i\lambda u|u|^{6})
=\lambda {\mathcal Im} (\nabla u, u \nabla (|u|^{6})).$$
Moreover we have
\begin{align}II&=-\frac 12 (\nabla |u|^2, \partial_t \nabla(|u|^3))=
-\frac{3}4 (\nabla |u|^2, \partial_t (\nabla(|u|^{2}) |u|))
\\\nonumber& =-\frac{3}4 \frac d{dt} (\nabla |u|^2, \nabla(|u|^{2}) |u|)
+ \frac{3}4  ( \partial_t \nabla |u|^2, \nabla(|u|^{2}) |u|)
\\\nonumber&=-\frac{3}4 \frac d{dt} (\nabla |u|^2, \nabla(|u|^{2}) |u|)
+ \frac{3}8  ( \partial_t |\nabla |u|^2|^2, |u|)
\\\nonumber&= -\frac{3}4 \frac d{dt} (\nabla |u|^2, \nabla(|u|^{2}) |u|)
+ \frac{3}8 \frac d{dt}  (|\nabla |u|^2|^2, |u|)
-\frac{3}8  (|\nabla |u|^2|^2, \partial_t |u|).
\end{align}

\end{proof}
\begin{lem}\label{lemnls}
Let $u$ be as in Theorem \ref{nls} and $U=\sup_{[0, T_{Max})} \|u(t)\|_{H^1}$,
then we have:
\begin{align}\label{ddt}\frac d{dt} &({\mathcal E}(u) + \|u\|_{L^2}^2)\lesssim 
U^{8} \ln^{3} (2+ \|u\|_{H^2})
\\\nonumber&+ U^3\|\Delta u\|_{L^2}^2
 \ln (2+\|u\|_{H^2}) + U^2 + U \|\Delta u\|_{L^2}^2, \quad \forall t\in [0, T_{max}).\end{align}
\end{lem}
\begin{proof}
Next we collect some useful inequalities satisfied by any solution $u$ of \eqref{NLS}:
\begin{align*}|{\mathcal Im} &(\nabla u, u \nabla (|u|^{6}))|
\lesssim \int |\nabla u|^2 \cdot |u|^{6} dx
\\\nonumber&\lesssim \|u\|_{H^1}^2 \|u\|_{L^\infty}^{6} 
\lesssim \|u\|_{H^1}^{8} \ln^{3} (2+ \|u\|_{H^2}) +\|u\|_{H^1}^{2} ,\end{align*}
where we used \eqref{improved}.
We also have
\begin{equation}\label{first}
\int |\nabla |u|^2|^2 \cdot |\partial_t |u|| dx\lesssim \int |\nabla u|^2 \cdot |u|^{6} dx
+ \int |\nabla u|^2 \cdot |\Delta u| \cdot |u|^{2} dx,\end{equation}
where we used the diamagnetic inequality $|\partial_t |u||\leq |\partial_t u|$
and the equation solved by $u$.
By combining the H\"older inequality, the logarithmic Sobolev embedding \eqref{improved}
and the Gagliardo-Nirenberg
inequality 
\begin{equation}\label{GNL4}\|\nabla u\|_{L^4}\lesssim \|\Delta u\|_{L^2}^\frac 12 \|\nabla u\|_{L^2}^\frac 12,\end{equation}
we can continue the estimate above as follows:
$$
...\lesssim \|u\|_{H^1}^{8} \ln^{3} (2+ \|u\|_{H^2}) +\|u\|_{H^1}^2+  \|\Delta u\|_{L^2}^2
\|u\|_{H^1}^{3} \ln (2+\|u\|_{H^2}) + \|\Delta u\|_{L^2}^2
\|u\|_{H^1}.$$
Finally notice that (by using the equation solved by $u$)
$$\int |\nabla u|^2 \cdot \partial_t (|u|^{3}) dx \lesssim 
\int |\nabla u|^2 \cdot |u|^{6}
+\int |\nabla u|^2 \cdot |\Delta u| \cdot |u|^{2} dx ,$$
and we can continue as in \eqref{first}.

\end{proof}

\noindent {\em Proof of Theorem \ref{nls}}
Assume by the absurd that $$T_{max}<\infty \hbox{ and } 
U = \sup_{t\in [0, T_{max})} \|u\|_{H^1}<\infty.$$
By elementary computations we get:
$$|(\nabla |u|^2, \nabla(|u|^{2}) |u|)|\lesssim (\int |\nabla u|^4 dx)^\frac 12
\cdot (\int |u|^6 dx)^\frac 12\lesssim U^4 \|\Delta u\|_{L^2},$$
where we used \eqref{GNL4}, and we also have 
$$|(\Delta u, u|u|^3)|\lesssim \|\Delta u\|_{L^2} \|u\|_{L^8}^4\lesssim U^4\|\Delta u\|_{L^2}.
$$
Hence
\begin{equation}\label{coercive}
\|u\|_{H^2}^2\lesssim {\mathcal E}(u) + \|u\|_{L^2}^2, \hbox{ for }  \|u\|_{H^2}>R=R(U)>0.
\end{equation}
Next recall that by definition of $T_{max}$ we have
$\|u(t)\|_{H^2}>R, \hbox{ } \forall t>\bar T\in (0, T_{max})$.
Hence by combining \eqref{coercive} with \eqref{ddt}
we get: 
\begin{align*}\|u(t)\|^2_{H^2}&\lesssim \|u(\bar T)\|_{H^2}^2
+ U^{8} \int_{\bar T} ^t  \ln^{3} (2+ \|u\|_{H^2}) dt+ U^{3}  \int_{\bar T}^t\|u\|_{H^2}^2 
\ln (2+\|u\|_{H^2}) dt \\\nonumber
&+ U \int_{\bar T}^t \|u\|_{H^2}^2 dt+U^2 (t-\bar T), \quad \forall t\in [\bar T, T_{max}).\end{align*}
We are in a position to conclude, arguing as in \cite{BG}, 
that $\sup_{t\in [0, T_{max})} \|u(t)\|_{H^2}<\infty$, and hence we get a contradiction
with the definition of $T_{max}$.

\hfill$\Box$

\section{The half-wave equation}
Along this section we use the notations:
$$|D_x|^s=(\sqrt{-\partial_x^2})^s,
(f, g)=\int_{\R} f\cdot \bar g \hbox{ } dx,
L^p=L^p(\R), H^k=H^k(\R).$$
We also introduce the energy
\begin{align}{\mathcal F}(u)= \|\partial_x u\|_{L^2}^2+2\lambda  {\mathcal Re} (|D_x| u, u|u|^{3})
-\frac 34 \lambda (||D_x|^\frac 12  (|u|^2)|^2, |u|)
\\\nonumber
+\lambda (|D_x|^\frac 12  |u|^2 -\bar u
|D_x|^\frac 12  u  - u |D_x|^\frac 12  \bar u  , |D_x|^\frac 12 (|u|^{3})).
\end{align}
The following proposition will be crucial in the sequel.
\begin{prop} (See \cite{kpv})\label{kpv}
We have the following estimate:
$$\||D_x|^s (fg) - g|D_x|^s f  - f |D_x|^s g\|_{L^p}\lesssim \||D_x|^{s_1} f\|_{L^q} 
\||D_x|^{s_2} f\|_{L^r},$$
where
$$\frac 1 p=\frac 1q+ \frac 1r, \hbox{ } 1<p,q,r<\infty, \hbox{ } 1>s=s_1+s_2>0, s_i\geq 0.$$ 
\end{prop}

\begin{lem}\label{lemsem} Let $u$ be as in Theorem \ref{nlssemirel}. Then we have the following identity:
\begin{align}\label{energymod}& \frac d{dt} ({\mathcal F}(u)+\|u\|_{L^2}^2)
=-2\lambda^2 {\mathcal Im} (|D_x| u, u|u|^{6}) +2\lambda (
||D_x|^\frac 12  u|^2, \partial_t (|u|^{3}))
\\\nonumber 
&+\lambda (|D_x|^\frac 12  \partial_t (|u|^2) - \partial_t(\bar u
|D_x|^\frac 12  u)  - \partial_t(u |D_x|^\frac 12  \bar u)  , |D_x|^\frac 12  (|u|^{3}))
\\\nonumber &+2\lambda {\mathcal Re} (
|D_x|^\frac 12  u, |D_x|^\frac 12 (u\partial_t (|u|^{3}))- |D_x|^\frac 12 u\partial_t (|u|^{3})
-u |D_x|^\frac 12 \partial_t (|u|^{3}))\\\nonumber&-\frac 34\lambda (||D_x|^\frac 12  (|u|^2)|^2,
\partial_t |u|)
+ \frac 32\lambda ( |D_x|^\frac 12  (|u|^2) ,   |D_x|^\frac 12 |u|  \partial_t (|u|^{2})) 
\\\nonumber&+ \frac 32\lambda ( |D_x|^\frac 12  (|u|^2) ,  |D_x|^\frac 12 (\partial_t (|u|^{2})|u|)
-|u||D_x|^\frac 12 \partial_t (|u|^{2}) -\partial_t (|u|^{2})|D_x|^\frac 12  |u|)
.\end{align}

\end{lem}
\begin{proof} Recall that $\frac d{dt} \|u\|_{L^2}^2=0$,
hence we shall treat $\frac d{dt} {\mathcal F}(u)$. In the sequel we assume
that the solution is regular enough in order to justify the following computations.
The proof in the case of lower regular solutions (i.e. $H^1$ solutions), can be done 
by a standard density argument. However we skip the details.\\
We make the following computation
\begin{align*}
\frac d{dt} &\|\partial_x u\|_{L^2}^2= 2 {\mathcal Re} (|D_x| \partial_t u, |D_x| u)
=2 {\mathcal Re} 
(|D_x| \partial_t u, i \partial_t u - \lambda u|u|^{3})\\\nonumber&=-2\lambda  {\mathcal Re} (|D_x| \partial_t u, u|u|^{3})=-2\lambda  \frac d{dt} {\mathcal Re} (|D_x| u, u|u|^{3})
+2\lambda  {\mathcal Re} (|D_x| u, \partial_t (u|u|^3)),\end{align*}
where we used the equation solved by $u$.
Next notice that
\begin{align*}
&{\mathcal Re} (|D_x| u, \partial_t (u|u|^{3}))\\\nonumber
&={\mathcal Re} (|D_x| u, \partial_t u |u|^{3}) + {\mathcal Re} (
|D_x|  u, u\partial_t (|u|^{3}))
=I+II.\end{align*}
Concerning $I$ we get (by using the equation solved by $u$)
$$I=-\lambda {\mathcal Im} (|D_x| u, u|u|^{6}),$$
and for $II$
we have
\begin{align}II&= {\mathcal Re} (
|D_x|^\frac 12  u, |D_x|^\frac 12 (u\partial_t (|u|^{3})))
\\\nonumber &= {\mathcal Re} (
|D_x|^\frac 12  u, |D_x|^\frac 12 u\partial_t (|u|^{3}))
+ {\mathcal Re} (
|D_x|^\frac 12  u,  u |D_x|^\frac 12 \partial_t (|u|^{3}))
\\\nonumber &+ {\mathcal Re} (
|D_x|^\frac 12  u, |D_x|^\frac 12 (u\partial_t (|u|^{3}))- |D_x|^\frac 12 u\partial_t (|u|^{3})
-u |D_x|^\frac 12 \partial_t (|u|^{3})),
\end{align}
that can be written as (recall $\partial_t (|u|^3)=\frac 32 \partial_t (|u|^2) |u|$)
\begin{align}&...
= {\mathcal Re} (
|D_x|^\frac 12  u, |D_x|^\frac 12 u\partial_t (|u|^{3}))
+ \frac 34 ( |D_x|^\frac 12  (|u|^2) ,   |D_x|^\frac 12 (\partial_t (|u|^{2}) |u|)) \\\nonumber&
- 
\frac 12 (|D_x|^\frac 12  (|u|^2) -\bar u
|D_x|^\frac 12  u  - u |D_x|^\frac 12  \bar u  , |D_x|^\frac 12 \partial_t (|u|^{3}))
\\\nonumber&+ {\mathcal Re} (
|D_x|^\frac 12  u, |D_x|^\frac 12 (u\partial_t (|u|^{3}))- |D_x|^\frac 12 u\partial_t (|u|^{3})
-u |D_x|^\frac 12 \partial_t (|u|^{3}))\\\nonumber&=II_1+II_2+II_3+II_4.
\end{align}
Next notice that 
\begin{align*}II_2&=\frac 34 ( |D_x|^\frac 12  (|u|^2) ,   |u||D_x|^\frac 12 \partial_t (|u|^{2}))
+ \frac 34 ( |D_x|^\frac 12  (|u|^2) ,   |D_x|^\frac 12 |u|  \partial_t (|u|^{2})) 
\\\nonumber& + \frac 34 ( |D_x|^\frac 12  (|u|^2) ,  |D_x|^\frac 12 (\partial_t (|u|^{2})|u|)
-|u||D_x|^\frac 12 \partial_t (|u|^{2}) -\partial_t (|u|^{2})|D_x|^\frac 12  |u|)
\end{align*}
and hence
\begin{align*}...&=\frac 38 (\partial_t ||D_x|^\frac 12  (|u|^2)|^2, |u|)
+ \frac 34 ( |D_x|^\frac 12  |u|^2 ,   |D_x|^\frac 12 |u| \partial_t (|u|^{2})) 
\\\nonumber&+ \frac 34 ( |D_x|^\frac 12  (|u|^2) ,  |D_x|^\frac 12 (\partial_t (|u|^{2})|u|)
-|u||D_x|^\frac 12 \partial_t (|u|^{2}) -\partial_t (|u|^{2})|D_x|^\frac 12  |u|)
\\\nonumber &
= \frac 38 \frac d{dt}( ||D_x|^\frac 12  (|u|^2)|^2, |u|)
\\\nonumber&
-  \frac 38 (||D_x|^\frac 12  |u|^2|^2,
\partial_t |u|) + \frac 34 ( |D_x|^\frac 12  (|u|^2) ,   |D_x|^\frac 12 |u|  \partial_t (|u|^{2})) 
\\\nonumber&+ \frac 34 ( |D_x|^\frac 12  (|u|^2) ,  |D_x|^\frac 12 (\partial_t (|u|^{2})|u|)
-|u||D_x|^\frac 12 \partial_t (|u|^{2}) -\partial_t (|u|^{2})|D_x|^\frac 12  |u|).
\end{align*}
Moreover we have  
\begin{align*}II_3&= - \frac 12  \frac d{dt}
(|D_x|^\frac 12  (|u|^2) -\bar u
|D_x|^\frac 12  u  - u |D_x|^\frac 12  \bar u  , |D_x|^\frac 12  (|u|^{3}))
\\\nonumber&+ \frac 12 (|D_x|^\frac 12 \partial_t  (|u|^2) - \partial_t(\bar u
|D_x|^\frac 12  u)  - \partial_t(u |D_x|^\frac 12  \bar u)  , |D_x|^\frac 12  (|u|^{3})).
\end{align*}

\end{proof}
\begin{lem}\label{lemsemir}
Let $u$ be as in Theorem \ref{nlssemirel} and let $U=\sup_{[0, T_{max})} \|u(t)\|_{H^{\frac 12}}$,
then we have
$$\frac d{dt} ({\mathcal F}(u)+\|u\|_{L^2}^2)
\lesssim (1+U)^6 \|u\|_{H^1}^2\ln (2+ \|u\|_{H^1}).$$
\end{lem}

\begin{proof} It follows by combining the estimates below with
Lemma \ref{lemsem}. More precisely we shall prove that
all the term on the r.h.s. in \eqref{energymod}
can be estimated by $(1+U)^6 \|u\|_{H^1}^2\ln (2+ \|u\|_{H^1})$.
First notice that $$|{\mathcal Im} (|D_x| u, u|u|^{6})|
\lesssim \|u\|_{H^1} \|u\|_{L^{14}}^7\lesssim \|u\|_{H^1}^2 
\|u\|_{H^\frac 12}^6\lesssim \|u\|_{H^1}^2 U^6.$$
On the other hand
$$|(
||D_x|^\frac 12  u|^2, \partial_t (|u|^{3}))|\lesssim 
\||D_x|^\frac 12  u\|_{L^4}^2
\|\partial_t u\|_{L^2}\|u\|_{L^\infty}^2,
$$
that by \eqref{improvedsemirel} and the following Gagliardo-Nirenberg inequality
\begin{equation}\label{GN}\||D_x|^\frac 12 u\|_{L^4}^2\lesssim \||D_x| u\|_{L^2}  \||D_x|^\frac 12 u\|_{L^2},\end{equation}
implies
$$...\lesssim \|u\|_{H^1} 
\|u\|_{H^{\frac 12}}^{3} \|\partial_t u\|_{L^2}
\ln (2+ \|u\|_{H^1}) + \|u\|_{H^1} 
\|u\|_{H^{\frac 12}} \|\partial_t u\|_{L^2}.$$
By looking at the equation solved by $u$
$$...\lesssim   \|u\|_{H^1} 
\|u\|_{H^{\frac 12}}^{3} (\|u\|_{H^1} + \|u\|_{L^8}^4)
\ln (2+ \|u\|_{H^1}) + \|u\|_{H^1} 
\|u\|_{H^{\frac 12}} (\|u\|_{H^1} + \|u\|_{L^8}^4).$$
Next notice that if we develop by the classical Leibniz rule the 
derivative with respect to the time variable and we apply twice 
 Proposition \ref{kpv} (where 
$s=s_1=\frac 12, s_2=0$, $p=\frac 43,q=2, r=4$) we get:
$$|(|D_x|^\frac 12  \partial_t(|u|^2) - \partial_t(\bar u
|D_x|^\frac 12  u)  - \partial_t(u |D_x|^\frac 12  \bar u)  , |D_x|^\frac 12  (|u|^{3}))|
$$
$$\lesssim  \|\partial_t u\|_{L^2} \||D_x|^\frac 12  u \|_{L^4} \||D_x|^\frac 12  (|u|^{3})\|_{L^4}
\lesssim \|\partial_t u\|_{L^2}  \| u\|_{H^1}^\frac 12\| u\|_{H^\frac 12}^\frac 12
\| |u|^{3}\|_{H^1}^\frac 12\| |u|^{3}\|_{H^\frac 12}^\frac 12$$
$$\lesssim \|\partial_t u\|_{L^2}  \| u\|_{H^\frac 12}
\| u\|_{H^1}\| u\|_{L^\infty}^2.$$
Notice that we have used \eqref{GN} and the property 
\begin{equation}\label{alg}\|v\cdot w\|_{H^s}
\lesssim \|v\|_{H^s}\|w\|_{L^\infty} +\|w\|_{H^s}\|v\|_{L^\infty}.
\end{equation}
 We conclude by using
\eqref{improvedsemirel}
and the equation solved by $u$.\\
Next we use again Proposition \ref{kpv} (where 
$s=s_1=\frac 12, s_2=0$, $p=\frac 43,q=2, r=4$),
$$|(
|D_x|^\frac 12  u, |D_x|^\frac 12 (u\partial_t (|u|^{3}))- |D_x|^\frac 12 u\partial_t (|u|^{3})
-u |D_x|^\frac 12 \partial_t (|u|^{3}))|$$
$$\lesssim \||D_x|^\frac 12  u\|_{L^4}^2
\|\partial_t (|u|^{3})\|_{L^2} \lesssim 
\|u\|_{H^1} \|u\|_{H^\frac 12} \|\partial_t u\|_{L^2} \|u\|_{L^\infty}^2,
$$ 
where we used \eqref{GN}.
We conclude by using \eqref{improvedsemirel} and the equation solved by $u$.\\
By H\"older inequality we get
$$|(||D_x|^\frac 12  (|u|^2)|^2,
\partial_t |u|)|\lesssim \||D_x|^\frac 12  (|u|^2)|\|_{L^4}^2 \|\partial_t u\|_{L^2} $$
$$\lesssim \|u^2\|_{H^\frac 12}  \|u^2\|_{H^1}  \|\partial_t u\|_{L^2}
\lesssim \|u\|_{H^\frac 12}  \|u\|_{H^1}\|u\|_{L^\infty}^2  \|\partial_t u\|_{L^2},
$$
where we have used \eqref{GN} and \eqref{alg}. We conclude as above. \\
Next we have the estimate
\begin{align}\nonumber|( |D_x|^\frac 12  (|u|^2) ,  |D_x|^\frac 12 |u| \partial_t (|u|^{2}))|&
\lesssim \| |D_x|^\frac 12  (|u|^2)\|_{L^4}\||D_x|^\frac 12 |u|\|_{L^4} \|\partial_t u\|_{L^2}
\|u\|_{L^\infty}\\\label{util} &\lesssim \|u^2\|_{H^1}^\frac 12 
\|u^2\|_{H^\frac 12}^\frac 12 \|u\|_{H^1}^\frac 12 
\|u\|_{H^\frac 12}^\frac 12 \|\partial_t u\|_{L^2}
\|u\|_{L^\infty},
\end{align}
where we used \eqref{GN}. By \eqref{alg} 
we get
$$...\lesssim \|u\|_{H^1}\|u\|_{H^\frac 12} \|u\|_{L^\infty}^2\|\partial_t u\|_{L^2},$$ 
and we conclude by using the equation solved by $u$ 
in conjunction with \eqref{improvedsemirel}.\\
Finally by Proposition \ref{kpv} and the H\"older inequality
we get the following estimate:
\begin{align*}& |( |D_x|^\frac 12  (|u|^2) ,  |D_x|^\frac 12 (\partial_t (|u|^{2})|u|)
-|u||D_x|^\frac 12 \partial_t (|u|^{2}) -\partial_t (|u|^{2})|D_x|^\frac 12  |u|)|
\\\nonumber & \lesssim \|( |D_x|^\frac 12  (|u|^2)\|_{L^4} \|\partial_t (|u|^{2})\|_{L^2}
\||D_x|^\frac 12 u\|_{L^4},
\end{align*}
and by \eqref{GN}
$$...\lesssim \|u^2\|_{H^1}^\frac 12 \|u^2\|_{H^\frac 12}^\frac 12\|u\|_{H^1}^\frac 12 \|u\|_{H^\frac 12}^\frac 12 \|\partial_t u\|_{L^2} \|u\|_{L^\infty},$$
which is precisely the term in \eqref{util}, hence we can conclude as above.

\end{proof}
 
\noindent {\em Proof of Theorem \ref{nlssemirel}}
 It is similar to the proof of Theorem \ref{nls}, provided that we use Lemma
 \ref{lemsemir} and we show that 
 $$|{\mathcal F}(u)+\|u\|_{L^2}^2 - \|u\|_{H^1}^2|\lesssim 
 C(U) (1+\|u\|_{H^1}) \ln^\frac 32 (2+ \|u\|_{H^1}).$$
This last fact follows from the following computations.
First notice that
$$|(|D_x| u, u|u|^{3})|\lesssim \||D_x|u\|_{L^2}\|u\|_{L^8}^4
\lesssim \|u\|_{H^1} \|u\|_{H^\frac 12}^4.
$$
Moreover we have 
$$
|(||D_x|^\frac 12  (|u|^2)|^2, |u|)|\lesssim
\||D_x|^\frac 12  (|u|^2)\|_{L^4}^2 \|u\|_{L^2}
$$$$\lesssim \|u^2\|_{H^\frac 12} \|u^2\|_{H^1} \|u\|_{L^2}
\lesssim \|u\|_{L^\infty}^2 \|u\|_{H^1}\|u\|_{H^\frac 12}^2,$$
where we used \eqref{GN} and \eqref{alg}.
We conclude by \eqref{improvedsemirel}.
Finally notice that
$$
|(|D_x|^\frac 12  |u|^2 -\bar u
|D_x|^\frac 12  u  - u |D_x|^\frac 12  \bar u  , |D_x|^\frac 12 (|u|^{3}))|
$$$$\lesssim \||D_x|^\frac 12  |u|^2\|_{L^2}  \||D_x|^\frac 12 (|u|^{3})\|_{L^2}
+ \|u\|_{L^2} \||D_x|^\frac 12  |u|\|_{L^4}  \||D_x|^\frac 12 (|u|^{3})\|_{L^4}
$$ 
and hence by \eqref{GN}
$$...\lesssim \|u^2\|_{H^\frac 12} 
\|u^3\|_{H^\frac 12} +
\|u\|_{L^2} \|u\|_{H^\frac 12}^\frac 12 \|u\|_{H^1}^\frac 12 \|u^3\|_{H^\frac 12}^\frac 12
\|u^3\|_{H^1}^\frac 12
$$$$\lesssim \|u\|_{H^\frac 12}^2 
\|u\|_{L^\infty}^3 + \|u\|_{L^2}  \|u\|_{H^1} \|u\|_{H^\frac 12}
\|u\|_{L^\infty}^2,
$$
where we used \eqref{alg}.
We conclude again by \eqref{improvedsemirel}.
 
\hfill$\Box$

\section{Proof of Corollary \ref{gwpHW}}

The case $\lambda=1$ follows by combining
the conservation of the energy $E_{HW,+}$ (which is positive definite)
with Theorem \ref{nlssemirel}.\\
Concerning the case $\lambda=-1$ it is sufficient to show that $\|u(t)\|_{\dot H^\frac 12}$ cannot blow--up in finite time
under the assumptions of Corollary \ref{gwpHW}.\\
Notice that by combining the conservation of the mass and the energy,
with the assumption
$E_{HW,-} (R) \|R\|_{L^2}^4 > 
E_{HW,-}(\varphi) \|\varphi \|_{L^2}^4$,
we get 
\begin{align}E_{HW,-} (R) \|R\|_{L^2}^4 >
E_{HW,-}(u(t)) \|u(t)\|_{L^2}^4 \\\nonumber= \frac 12 \||D_x|^\frac 12 u(t)\|_{L^2}^2\|u(t)\|_{L^2}^4
- \frac 15 \|u(t)\|_{L^5}^5 \|u(t)\|_{L^2}^4.\end{align}
By the following Gagliardo--Nirenberg inequality
\begin{equation}\label{GNdisp} 
\|g\|_{L^5(\R)} \leq C_{GN} \||D_x|^\frac 12 g\|_{L^2(\R)}^\frac 35\|g\|_{L^2(\R)}^\frac 25
\end{equation}
we get
$$...\geq \frac 12 \big ( \||D_x|^\frac 12 u(t)\|_{L^2}\|u(t)\|_{L^2}^2
\big )^2 - \frac 15 C_{GN}^5  \big ( \||D_x|^\frac 12 u(t)\|_{L^2}\|u(t)\|_{L^2}^2
\big )^3.
$$
Hence $\||D_x|^\frac 12 u(t)\|_{L^2}\|u(t)\|_{L^2}^2$ belongs to the sublevel
$${\mathcal A}=\{x\in \R^+| f(x)<E_{HW,-} (R) \|R\|_{L^2}^4\},$$
where $f(x)=\frac 12 x^2 - \frac 15 C_{GN}^5 x^3$. 
Next we denote by $x_{max}>0$ the unique point where the maximum of $f$ is achieved
on $(0, \infty)$. 
We claim that
\begin{equation}\label{claim}x_{max}= \||D_x|^\frac 12 R\|_{L^2}\|R\|_{L^2}^2
\hbox{ and }
f(x_{max})= E_{HW,-} (R) \|R\|_{L^2}^4.\end{equation}
If this is the case then we get
$${\mathcal A}= {\mathcal A}_1\cup {\mathcal A}_2,$$
where
$${\mathcal A}_1=(0,  \||D_x|^\frac 12 R\|_{L^2}\|R\|_{L^2}^2)
\hbox{ and } {\mathcal A}_2= (\||D_x|^\frac 12 R\|_{L^2}\|R\|_{L^2}^2, \infty),
$$
and we conclude by a continuity argument.\\
First notice that by the analysis of $f'$ we get
$$x_{max}= \frac 5{3 C_{GN}^5}$$
and also since $R$ is an optimizer for 
\eqref{GNdisp}, then 
$$E_{HW,-} (R) \|R\|_{L^2}^4= \frac 12 \big ( \||D_x|^\frac 12 R\|_{L^2}\|R\|_{L^2}^2
\big )^2 - \frac 15 C_{GN}^5  \big ( \||D_x|^\frac 12 R\|_{L^2}\|R\|_{L^2}^2
\big )^3.$$
Hence \eqref{claim} follows provided that we prove
\begin{equation}\label{last}
\frac 5{3 C_{GN}^5}= \||D_x|^\frac 12 R\|_{L^2}\|R\|_{L^2}^2.\end{equation}
To prove this fact notice that
since $R$ is an optimizer
for \eqref{GNdisp} we get
$$\frac d{dt} ( \int |R+t\varphi|^5 dx - C_{GN}^5 \||D_x|^\frac 12 (R+t\varphi)\|_{L^2}^3
\|R+t\varphi\|_{L^2}^2)
_{t=0}=0, \hbox{ } \forall \varphi\in H^\frac 12$$
and hence by direct computations
it implies 
$$-(3 C_{GN}^5 \|R\|_{L^2}^2 \||D_x|^\frac 12R\|_{L^2}) |D_x| R - (2 C_{GN}^5 
\||D_x|^\frac 12R\|_{L^2}^3) R + 5R^4=0.$$
Since $R$ solves \eqref{normal}
we deduce that
\begin{equation}\label{34}3 C_{GN}^5 \|R\|_{L^2}^2 \||D_x|^\frac 12R\|_{L^2}=
2 C_{GN}^5 
\||D_x|^\frac 12R\|_{L^2}^3=5\end{equation}
and hence we get \eqref{last}.
Notice that  \eqref{34} follows by the fact
that $R$ cannot be a solution
to $|D_x| R+ a R=bR^4$ unless $a=b=1$.
In fact if it not the case then, since $R$ solves \eqref{normal}, we would get 
$(b-1)R^4=(a-1)R$ that implies $R$ is a constant.

\end{document}